\documentclass[10pt]{amsart}

\usepackage{amsmath,amssymb,amsthm}

\usepackage{graphicx}
\usepackage{pinlabel}
\usepackage{xcolor}

\usepackage[numbers,sort&compress]{natbib}
\usepackage{hyperref}

\theoremstyle{plain}
\newtheorem{theorem}{Theorem}
\newtheorem{lemma}[theorem]{Lemma}

\theoremstyle{definition}

\theoremstyle{remark}
\newtheorem{remark}[theorem]{Remark}

\title{Density of Length Ratios on Teichm\"uller Space}

\author{Nhat Minh Doan}
\address{Institute of Mathematics,
Vietnam Academy of Science and Technology,
Hanoi, Vietnam}
\email{dnminh@math.ac.vn}

\author{Hanh Vo}
\address{Department of Mathematics,
Purdue University, 
Indiana, USA}
\email{vo84@purdue.edu}


\date{}

\begin{document}

\begin{abstract}
We study length ratios of simple closed geodesics on
finite-type hyperbolic surfaces of positive genus.
We allow cusps and fix the geodesic boundary lengths. 
Let $d$ be the dimension of the corresponding
Teichm\"uller space $\mathcal T_{\mathbf L}(S)$.
We prove that for every positive vector $(r_0,\ldots,r_d)$,
the set of surfaces carrying $d+1$ distinct essential
interior simple closed geodesics with lengths proportional
to this vector is dense in $\mathcal T_{\mathbf L}(S)$. 
In particular, one-holed tori with three distinct
simple closed geodesics of equal length are dense.
The proof uses high powers of Dehn twists, asymptotics
of length differentials, and the Brouwer fixed-point theorem.
\end{abstract}

\maketitle

\section{Introduction}

Geodesic lengths exhibit both rigidity and flexibility phenomena.
On the rigidity side, the marked length spectrum determines
the metric on a closed negatively curved surface
\cite{MR1038361}.
Even the ordering of curve lengths can determine a hyperbolic
metric \cite{MR2431011,parlier2025orderingcurvessurfaces}.
On the other hand, the unmarked length spectrum does not
always determine the surface, as shown by examples of
non-isometric isospectral surfaces \cite{MR850750}.
Length coincidences are also abundant.
Randol proved that the length spectrum of every closed
hyperbolic surface has unbounded multiplicity
\cite{MR553396}.

Let $S$ be a connected oriented surface of finite type
with negative Euler characteristic. Fix the geodesic
boundary lengths $\mathbf L$, and let
$\mathcal T_{\mathbf L}(S)$ denote the corresponding
Teichm\"uller space.
For $X\in\mathcal T_{\mathbf L}(S)$, let $\ell_X(\alpha)$
denote the length of the geodesic representative of an
essential closed curve $\alpha$.
For two distinct essential interior simple closed curves
$\alpha$ and $\beta$, consider the equal-length locus
\[
E(\alpha,\beta)
=\{X\in\mathcal T_{\mathbf L}(S):
\ell_X(\alpha)=\ell_X(\beta)\}.
\]
For a one-holed torus with fixed boundary length,
McShane and Parlier proved that each $E(\alpha,\beta)$
is a simple path \cite[Theorem~1.3]{MR2431011}.
They further expressed a strong suspicion that the
locus of one-holed tori with three distinct simple
closed geodesics of equal length should be nowhere
dense \cite[page 1884]{MR2431011}.

Surprisingly, this threefold-equality locus is dense.
In fact, this is a special case of a more general
result on prescribed length ratios. 
For our main result, assume that $S=S_{g,n,b}$ has genus
$g\geq 1$, with $n$ cusps and $b$ geodesic boundary
components, where $2g-2+n+b>0$.
The real dimension of $\mathcal T_{\mathbf L}(S)$ is
$d=6g-6+2n+2b$.

\begin{theorem}\label{thm:main_intro}
For every positive vector
$(r_0,\ldots,r_d)\in\mathbb R_{>0}^{d+1}$,
the set of surfaces $Y\in\mathcal T_{\mathbf L}(S)$
admitting $d+1$ distinct essential interior simple
closed curves $\gamma_0,\ldots,\gamma_d$ satisfying
\[
\frac{\ell_Y(\gamma_0)}{r_0}
=\frac{\ell_Y(\gamma_1)}{r_1}
=\cdots
=\frac{\ell_Y(\gamma_d)}{r_d}
\]
is dense in $\mathcal T_{\mathbf L}(S)$.
\end{theorem}

The proof combines high powers of Dehn twists with
a local coordinate argument.
Near any given surface $X$, we choose $d+1$ simple
closed curves whose logarithmic length ratios give
local coordinates.
We then twist suitable transversal curves around them,
with twist powers chosen to approximate the prescribed
ratios.
The resulting normalized length-ratio maps converge
uniformly to the local coordinate map near $X$.
The Brouwer fixed-point theorem gives a nearby surface
where the ratios hold exactly.

\subsection*{Structure of the paper}

Section \ref{section-preliminaries} establishes the
one-holed torus length formula and the estimates for
high powers of Dehn twists.
Section \ref{section-higher-dimensional} proves the
main theorem using logarithmic length covectors
and the Brouwer fixed-point theorem.
We conclude with a density result
along coordinate lines in the Teichm\"uller space
of a one-holed torus.

\subsection*{Acknowledgments}
The authors thank Hugo Parlier for his encouragement.

\section{Preliminaries}\label{section-preliminaries}

In this section, we establish the estimates needed
for the proof of the main theorem. We first derive
an exact length formula on a one-holed torus, followed
by uniform estimates for high powers of Dehn twists.
We then establish a topological criterion ensuring
the distinctness of the resulting curves and derive
asymptotic expansions for their length differentials.

\subsection{Lengths on a one-holed torus}

Let $S$ be a one-holed torus with boundary length
$L\ge0$, where $L=0$ corresponds to a cusp. 
Let $\alpha$ and $\beta$ be a pair of essential simple closed curves on $S$ intersecting exactly once, $i(\alpha, \beta) = 1$. We define the coordinates $(a, \tau)$ for any hyperbolic surface $X \in \mathcal{T}_L(S)$ as follows:
\begin{itemize}
\item $a = \ell_X(\alpha)$ is the hyperbolic length of the base curve $\alpha$.
\item $\tau\in\mathbb R$ is the length-valued
Fenchel-Nielsen twist parameter about $\alpha$,
normalized so that $\tau=0$ when the two endpoints
of the common perpendicular between the
$a$-boundaries are identified. At zero twist,
the geodesic representative of $\beta$ meets
$\alpha$ orthogonally. 
\end{itemize}

Cutting $X$ along $\alpha$ yields a pair of pants with two boundary components of length $a$ and one boundary component of length $L$. Let $h(a,L)$ be the length of the unique shortest geodesic seam connecting the two $a$-boundaries (see Figure \ref{fig:pants}). By standard hyperbolic trigonometry for right-angled hexagons  (see \cite[Theorem 2.4.1(i)]{MR2742784}), this seam length satisfies:
\[
\cosh h(a,L) = \frac{\cosh(L/2) + \cosh^2(a/2)}{\sinh^2(a/2)}
\implies
\cosh \frac{h(a,L)}{2} = \sqrt{\frac{\cosh a + \cosh(L/2)}{\cosh a - 1}}.
\]
To simplify the length formulas for twisted curves and their derivatives in Section \ref{section-higher-dimensional}, we introduce the following notation:
\begin{equation}\label{eq:f_def}
F(a,L) = \cosh \frac{h(a,L)}{2} = \sqrt{\frac{\cosh a + \cosh(L/2)}{\cosh a - 1}}, \qquad f(a,L) = \log F(a,L).
\end{equation}

\begin{figure}
    \centering
    \labellist 
    \pinlabel {$a$} [r] at 20 190
    \pinlabel {$a$} [r] at 270 210
      
    \pinlabel {$a/2$} [b] at 350 250
    \pinlabel {$a/2$} [b] at 350 170

    \pinlabel {$a/2$} [b] at 610 250
    \pinlabel {$a/2$} [b] at 610 170 

    \pinlabel {$h(a,L)$} [b] at 140 145

    \pinlabel {$h(a,L)$} [b] at 479 215

    \endlabellist
    \includegraphics[width=0.6\linewidth]{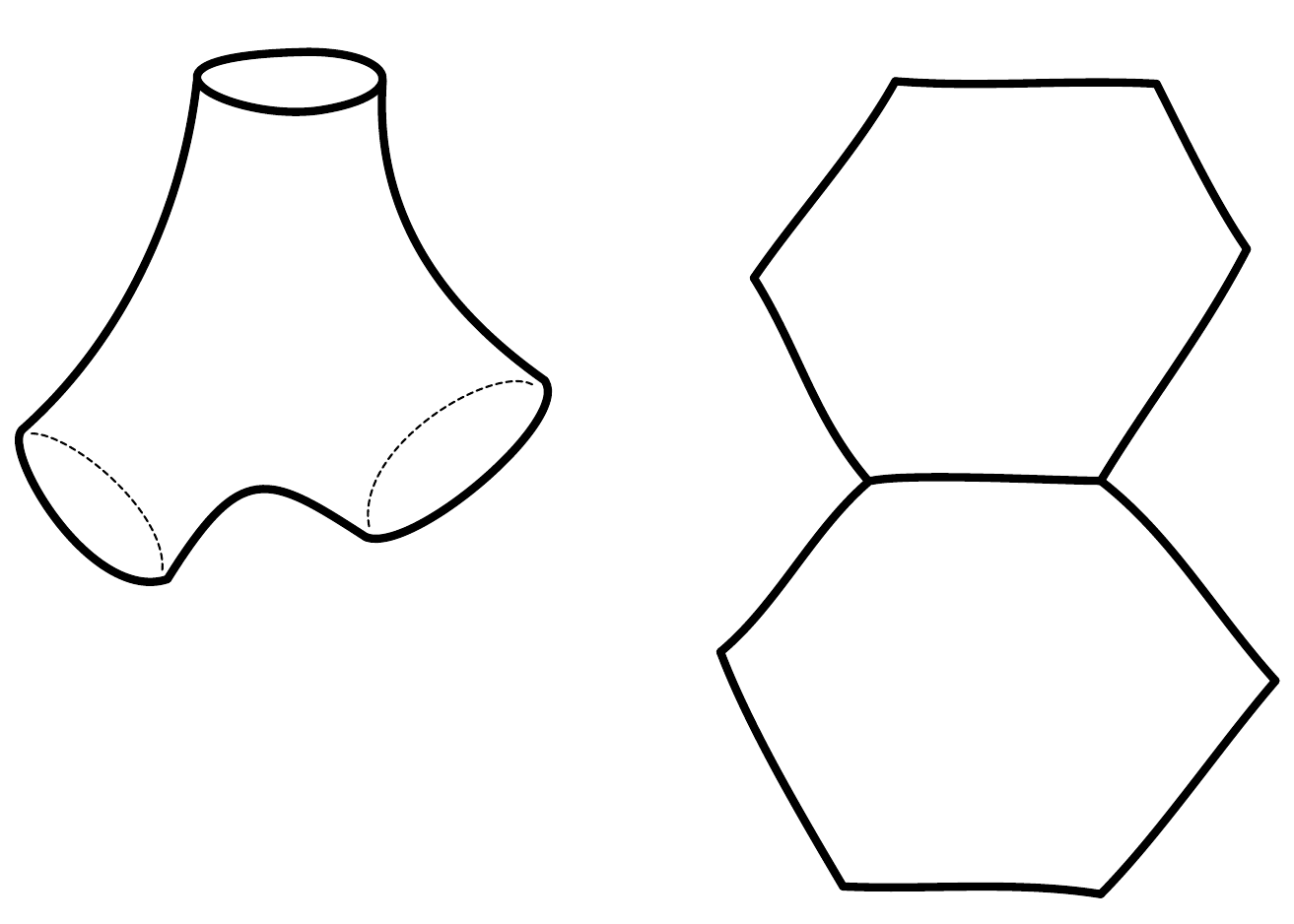}

    \caption{$h(a,L)$ is the length of the unique shortest geodesic seam connecting the two $a$-boundaries}
    \label{fig:pants}
\end{figure}

Throughout this paper, following standard conventions for length-valued Fenchel-Nielsen coordinates, applying a positive Dehn twist to a curve around $\alpha$ increases its effective twist displacement relative to the orthogonal seam by exactly $+\ell_X(\alpha)$.

\begin{lemma}\label{lemma:gluing-identity}
Let $X = (a, \tau) \in \mathcal{T}_L(S)$. The hyperbolic length of the reference curve $\beta$ on $X$ is given by:
\[
\cosh\frac{\ell_X(\beta)}{2} = F(a,L)\cosh\frac{\tau}{2}.
\]
More generally, for any integer $n \in \mathbb{Z}$, the curve $\beta_n = T_\alpha^n(\beta)$ obtained by applying $n$ Dehn twists to $\beta$ around $\alpha$ has length:
\[
\ell_X(\beta_n) = 2\operatorname{arcosh}\left( F(a,L)\cosh\frac{\tau + na}{2} \right).
\]
\end{lemma} 

\begin{proof}
Cut $X$ along $\alpha$, obtaining a hyperbolic pair of pants with
boundary lengths $a,a,L$.  Choose two boundary lifts
$\widetilde\alpha_-$ and $\widetilde\alpha_+$ in the universal cover
of this pair of pants which are joined by a lift of the orthogonal
seam.  Let $p_-$ and $p_+$ be the endpoints of this common
perpendicular, whose length is $h=h(a,L)$, and let $M$ be its
midpoint.

We first fix the signs in the arclength coordinates.  Choose a unit
tangent vector $v_-$ to $\widetilde\alpha_-$ at $p_-$, and let
$v_+$ be its parallel transport along the common perpendicular to
$p_+$.  Write $x_-(s)$ and $x_+(s)$ for the points at signed
arclength $s$ from $p_-$ and $p_+$ in the directions $v_-$ and
$v_+$, respectively.  In these compatible coordinates, the
zero-twist gluing identifies $x_-(s)$ with $x_+(s)$, and the length-valued twist $\tau$ identifies
$x_-(s)$ and $x_+(s+\tau)$. 
For $s\in\mathbb R$, let $\gamma_s$ be the geodesic segment
joining $x_-(s)$ to $x_+(s+\tau)$, and set
\[
D(s)=d_{\mathbb H^2}\bigl(x_-(s),x_+(s+\tau)\bigr).
\]
The segment $\gamma_s$ projects, after gluing, to a loop in the
free homotopy class of $\beta$. A fundamental segment of a lift of the geodesic representative of $\beta$ joins
$x_-(s)$ to $x_+(s+\tau)$ for some $s$.
Therefore,
\[
\ell_X(\beta)=\min_{s\in\mathbb R}D(s).
\] 
Let $\iota$ be the half-turn about $M$; that is,
$M$ is the midpoint of the geodesic segment joining
$P$ to $\iota(P)$ for every $P\in\mathbb H^2$.
Since $M$ is the midpoint of the common perpendicular,
$\iota$ exchanges $\widetilde{\alpha}_-$ and
$\widetilde{\alpha}_+$ and sends $p_-$ to $p_+$.
With our choice of signed arclength orientations,
\[
\iota(x_-(s))=x_+(-s),
\qquad
\iota(x_+(s))=x_-(-s).
\]
Since $\iota$ is an isometry and distance is symmetric,  
\[
\begin{aligned}
D(s)
&=d_{\mathbb H^2}\bigl(x_-(s),x_+(s+\tau)\bigr)\\
&=d_{\mathbb H^2}
 \bigl(\iota(x_-(s)),\iota(x_+(s+\tau))\bigr)\\
&=d_{\mathbb H^2}\bigl(x_+(-s),x_-(-s-\tau)\bigr)\\
&=d_{\mathbb H^2}\bigl(x_-(-s-\tau),x_+(-s)\bigr)\\
&=D(-s-\tau).
\end{aligned}
\]
Hence $D$ is symmetric about $s=-\tau/2$.  
Since $D$ is convex, its minimum occurs at
\[
s=-\frac{\tau}{2}.
\]
The endpoints of the minimizing segment are therefore
\[
A=x_-\left(-\frac\tau2\right),
\qquad
B=x_+\left(\frac\tau2\right).
\]
The half-turn sends $A$ to $B$, so $M$ is the midpoint of the
geodesic segment $AB$.  In particular,
\[
MA=\frac{\ell_X(\beta)}2.
\]

\begin{figure}
    \centering
    \labellist  
    \pinlabel {\color{red}$\beta$} [r] at 275 265
    
    \pinlabel {$M$} [r] at 275 195
      
    \pinlabel {$p_+$} [b] at 380 190

    \pinlabel {$p_-$} [b] at 120 190
    \pinlabel {$A$} [b] at 110 145
    \pinlabel {$B$} [b] at 370 260
 
    \endlabellist
    \includegraphics[width=0.35\linewidth]{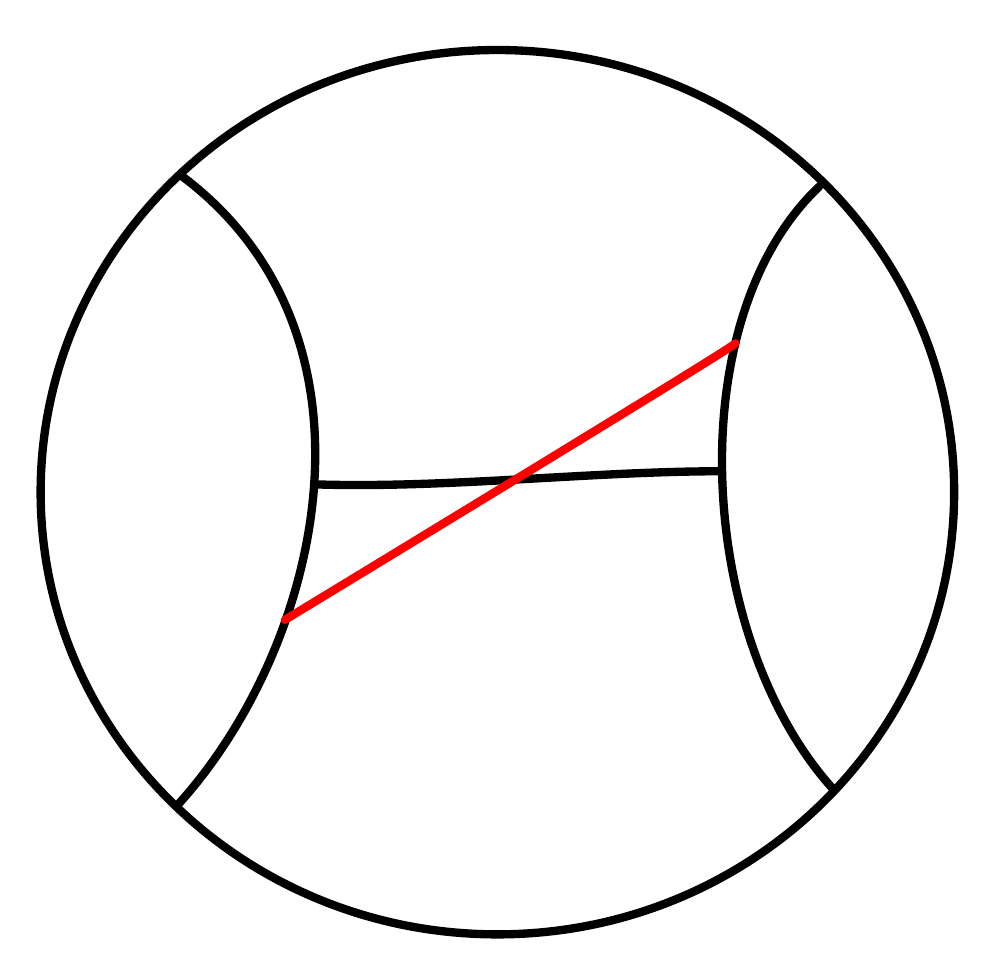}

    \caption{Computing $\ell_X(\beta)$}
    \label{fig:length-beta}
\end{figure}

The triangle with vertices $M,p_-,A$ is right-angled at $p_-$,
because the common perpendicular meets $\widetilde\alpha_-$
orthogonally.  Its two legs have lengths $h/2$ and $|\tau|/2$,
while its hypotenuse has length $\ell_X(\beta)/2$ (see Figure \ref{fig:length-beta}).  The
hyperbolic right-triangle cosine law gives
\[
\cosh\frac{\ell_X(\beta)}2
=\cosh\frac h2\cosh\frac{|\tau|}2
=F(a,L)\cosh\frac\tau2.
\]

Finally, under our twist convention, the curve
$\beta_n=T_\alpha^n(\beta)$ has effective twist displacement
$\tau+na$.  Applying the same identity with $\tau+na$ in place
of $\tau$ yields
\[
\cosh\frac{\ell_X(\beta_n)}2
=F(a,L)\cosh\frac{\tau+na}2.
\]
This completes
the proof.
\end{proof}

\subsection{Uniform estimates for high powers of Dehn twists}

Let $S$ be a finite-type surface of negative Euler characteristic that is not a pair of pants. Let $\mathcal{M}$ denote either the full Teichm\"uller space $\mathcal{T}(S)$ or a fixed-boundary subspace $\mathcal{T}_{\mathbf{L}}(S)$, where boundary lengths are fixed to a vector $\mathbf{L}$. Let $B \subset \mathcal{M}$ be a compact set.

\begin{lemma}\label{lemma:universal_collar_squeeze}
Let $\eta$ and $\xi$ be essential simple closed curves on a surface $S$ intersecting transversely with geometric intersection number $I = i(\eta, \xi) > 0$. Let $B\subset\mathcal M$ be compact. There exists a constant $C > 0$, depending only on $B$, $\eta$, and $\xi$, such that
\[
\left| \ell_Y(T_\eta^N(\xi)) - N \cdot I \cdot \ell_Y(\eta) \right| \le C
\]
for every $Y \in B$ and every integer $N \ge 1$.
\end{lemma}

\begin{proof}
We first establish the bounds for a fixed hyperbolic surface
$Y\in B$, and then show that the constants can be chosen
uniformly over $B$.

\medskip
\noindent\textbf{1. Bounds for a fixed metric $Y$.} 
Fix $Y\in B$. For the upper bound, take the geodesic
representative of $\xi$ and insert $N$ turns around $\eta$
at each of its $I$ crossings. This gives
\[
\ell_Y(T_\eta^N(\xi))
\le NI\ell_Y(\eta)+\ell_Y(\xi).
\]

For the lower bound, we follow the stairstep construction
of Slegers \cite[Lemmas~2.2--2.4]{MR4444143}.
Although Slegers works with closed surfaces, the argument
also applies to surfaces with geodesic boundary or cusps:
their universal covers are isometric to convex subsets of
$\mathbb H^2$, so we may use the fixed hyperbolicity
constant $\delta_{\mathbb H}$.

Take the geodesic representatives of $\eta$ and $\xi$.
Their $I$ intersection points cut $\xi$ into arcs
$\xi_Y^1,\ldots,\xi_Y^I$. For each $j$, let $A_j(Y)$
be the shortest geodesic arc in the corresponding relative
homotopy class, with endpoints free to slide along $\eta$.
By the first variation of length, $A_j(Y)$ meets
$\eta$ orthogonally. The curve $\xi$ is freely homotopic
to a piecewise-geodesic loop
\[
\omega_0(Y)
=A_1(Y)B_{1,0}(Y)\cdots A_I(Y)B_{I,0}(Y),
\]
where each $B_{j,0}(Y)$ lies along $\eta$.
Orient the relevant lifts of $\eta$ so that a positive
Dehn twist contributes positive signed displacement.
Let $b_j(Y)$ denote the signed displacement of
$B_{j,0}(Y)$, and set
\[
c(Y)=\max_{1\le j\le I}|b_j(Y)|.
\]

After $N$ positive Dehn twists, $T_\eta^N(\xi)$ is
represented by
\[
\omega_N(Y)
=A_1(Y)B_{1,N}(Y)\cdots A_I(Y)B_{I,N}(Y),
\]
where
\[
b_{j,N}(Y)=b_j(Y)+N\ell_Y(\eta).
\]
Consequently,
\[
\begin{aligned}
\ell_Y(B_{j,N}(Y))
&=|b_{j,N}(Y)|\\
&\ge N\ell_Y(\eta)-|b_j(Y)|\\
&\ge N\ell_Y(\eta)-c(Y).
\end{aligned}
\]
For $N$ sufficiently large, each $b_{j,N}(Y)$ is
positive and greater than $\ell_Y(\eta)$.
As in \cite[Lemma~2.2]{MR4444143}, the successive
right angles have alternating orientations, so every
lift of $\omega_N(Y)$ is a stairstep path.

Choose a vertex $x$ of a lift of $\omega_N(Y)$, and
let $g$ be the corresponding deck transformation.
Concatenating $m$ successive translates gives a
stairstep path from $x$ to $g^m x$ with $2Im$
geodesic segments. By \cite[Lemma~2.4]{MR4444143},
the error is at most $4(2Im-1)\delta_{\mathbb H}$.
Hence
\[
\begin{aligned}
d(x,g^m x)
&\ge m\sum_{j=1}^I
\left(
\ell_Y(A_j(Y))+\ell_Y(B_{j,N}(Y))
\right)\\
&\qquad -4(2Im-1)\delta_{\mathbb H}.
\end{aligned}
\]
Since $g$ represents the free homotopy class of
$T_\eta^N(\xi)$, the translation-length formula gives
\[
\ell_Y(T_\eta^N(\xi))
=\lim_{m\to\infty}\frac{d(x,g^m x)}{m}.
\]
Setting $C_{\mathrm{step}}=8I\delta_{\mathbb H}$,
we obtain
\begin{align*}
\ell_Y(T_\eta^N(\xi))
&\ge \sum_{j=1}^I
\left(
\ell_Y(A_j(Y))+\ell_Y(B_{j,N}(Y))
\right)
-C_{\mathrm{step}}\\
&\ge \sum_{j=1}^I
\ell_Y(B_{j,N}(Y))-C_{\mathrm{step}}\\
&\ge \sum_{j=1}^I
\left(N\ell_Y(\eta)-c(Y)\right)
-C_{\mathrm{step}}\\
&=NI\ell_Y(\eta)-Ic(Y)-C_{\mathrm{step}}.
\end{align*}
Thus, for each fixed $Y$, the desired estimate holds
for all sufficiently large $N$.  

\medskip
\noindent\textbf{2. Uniformity over the compact set $B$.} 
The relative homotopy classes
defining the arcs $A_j(Y)$ are fixed by the marking.
Their orthogeodesic representatives and the signed
displacements $b_j(Y)$ vary continuously with $Y$.
Likewise, the lengths $\ell_Y(\eta)$ and $\ell_Y(\xi)$
are continuous. Since $B$ is compact, we may define
\[
\begin{aligned}
c_0&=\max_{Y\in B}\max_j |b_j(Y)|<\infty,\\
L_\xi&=\max_{Y\in B}\ell_Y(\xi)<\infty,\\
a_0&=\min_{Y\in B}\ell_Y(\eta)>0,\\
L_\eta&=\max_{Y\in B}\ell_Y(\eta)<\infty.
\end{aligned}
\]

Choose an integer $N_0\ge2$ such that
\[
N_0a_0>c_0+L_\eta.
\]
For every $Y\in B$, $N\ge N_0$, and every $j$,
we have
\[
\begin{aligned}
b_{j,N}(Y)
&=b_j(Y)+N\ell_Y(\eta)\\
&\ge Na_0-c_0\\
&\ge N_0a_0-c_0\\
&>L_\eta\ge\ell_Y(\eta).
\end{aligned}
\]
Thus every $B_{j,N}(Y)$ winds at least once around
$\eta$ in the positive twist direction. In particular,
the stairstep construction above applies for all
$Y\in B$ and $N\ge N_0$. Since $c(Y)\le c_0$,
the lower bound becomes
\[
\ell_Y(T_\eta^N(\xi))
\ge NI\ell_Y(\eta)-Ic_0-C_{\mathrm{step}}
\]
for all $Y\in B$ and $N\ge N_0$.

For each of the finitely many integers $1\le N<N_0$,
the function
\[
Y\longmapsto
\ell_Y(T_\eta^N(\xi))-NI\ell_Y(\eta)
\]
is continuous and hence bounded on $B$.
We can therefore choose
\[
C=\max\left\{
Ic_0+C_{\mathrm{step}},\ L_\xi,\
\max_{\substack{Y\in B\\1\le N<N_0}}
\left|\ell_Y(T_\eta^N(\xi))-NI\ell_Y(\eta)\right|
\right\}.
\]
Combining the upper and lower bounds gives
\[
\left|
\ell_Y(T_\eta^N(\xi))-NI\ell_Y(\eta)
\right|\le C
\]
for every $Y\in B$ and every integer $N\ge1$.
\end{proof}

\subsection{Intersection growth and distinctness}

We recall the following intersection estimate from
\cite[Proposition~A.1]{MR3053012}.
Although stated for closed surfaces, its proof uses
only local estimates in annular neighborhoods of
the twisting curves. Hence the same bound holds
for finite-type surfaces with cusps or geodesic
boundary, provided all curves are essential
and interior.

\begin{lemma}[{\cite[Proposition A.1]{MR3053012}}]\label{lemma:intersection_bound}
For any essential interior simple closed curves $\eta, \xi, \sigma$ on a finite-type surface and every positive integer $N$, the geometric intersection number satisfies:
\[
\left| i(T_\eta^N(\xi), \sigma) - N i(\eta,\xi)i(\eta,\sigma) \right| \leq i(\xi,\sigma).
\]
\end{lemma}

We use this estimate to show that sufficiently
large Dehn twists about distinct base curves
produce distinct target curves.

\begin{lemma}[Distinctness of target curves]\label{lemma:distinct_targets}
Let $X$ be a hyperbolic surface. Let $\eta_0$ and $\eta_1$ be distinct essential interior simple closed curves. For $m \in \{0, 1\}$, let $\xi_m$ be an essential interior simple closed curve with $I_m = i(\eta_m, \xi_m) > 0$. Given  $r_0, r_1 > 0$, define  
\[
N_m(k) = \left\lfloor \frac{k r_m}{I_m \ell_X(\eta_m)} \right\rfloor.
\]
Then for all sufficiently large $k$, the twisted target curves $\gamma_0^{(k)} = T_{\eta_0}^{N_0(k)}(\xi_0)$ and $\gamma_1^{(k)} = T_{\eta_1}^{N_1(k)}(\xi_1)$ are distinct.
\end{lemma}

\begin{proof}
By applying Lemma \ref{lemma:intersection_bound} to each target curve and dividing by the scaling factor $N_m(k)I_m$, we obtain for any essential interior test curve $\sigma$:
\[
\left| \frac{i(\gamma_m^{(k)},\sigma)}{N_m(k)I_m} - i(\eta_m,\sigma) \right| \leq \frac{i(\xi_m,\sigma)}{N_m(k)I_m}.
\]
As $k \to \infty$, we have $N_m(k) \to \infty$, so the right side vanishes, giving the limit:
\[
\frac{i(\gamma_m^{(k)},\sigma)}{N_m(k)I_m} \longrightarrow i(\eta_m,\sigma).
\]
Suppose for contradiction that $\gamma_0^{(k)} = \gamma_1^{(k)}$ for infinitely many $k$. Along this subsequence, their intersection numbers with any test curve $\sigma$ are equal. Thus,
\[
\frac{i(\gamma_0^{(k)},\sigma)}{N_0(k)I_0} = \frac{i(\gamma_1^{(k)},\sigma)}{N_0(k)I_0} = \left( \frac{N_1(k)I_1}{N_0(k)I_0} \right) \frac{i(\gamma_1^{(k)},\sigma)}{N_1(k)I_1}.
\]
By the definition of $N_m(k)$, we can write
\[
N_m(k)
=\frac{kr_m}{I_m\ell_X(\eta_m)}-\delta_m(k),
\qquad 0\le\delta_m(k)<1,
\]
where $\delta_m(k)$ is the fractional
rounding error. Consequently, as $k\to\infty$,
\[
\frac{N_1(k)I_1}{N_0(k)I_0} = \frac{ \frac{k r_1}{\ell_X(\eta_1)} - \delta_1(k)I_1 }{ \frac{k r_0}{\ell_X(\eta_0)} - \delta_0(k)I_0 } \longrightarrow \frac{r_1\ell_X(\eta_0)}{r_0\ell_X(\eta_1)} =: \lambda > 0.
\]
Passing to the limit gives $i(\eta_0,\sigma) = \lambda i(\eta_1,\sigma)$ for every test curve $\sigma$. We obtain a contradiction by considering two cases:
\begin{enumerate}
\item \textbf{Case 1: $i(\eta_0,\eta_1) > 0$.} Take $\sigma = \eta_0$. The identity becomes $i(\eta_0, \eta_0) = \lambda i(\eta_1, \eta_0)$. The left side is $0$, while the right side is positive (since $\lambda > 0$), yielding a contradiction.

\item \textbf{Case 2: $i(\eta_0,\eta_1)=0$.}
Cut the surface along $\eta_1$, and let $S'$ be the component containing $\eta_0$. The curve
$\eta_0$ remains an essential, non-peripheral interior curve of $S'$. If $S'$ were a pair of
pants, every essential simple closed curve in $S'$ would be isotopic to a boundary component; in
the original surface this would make $\eta_0$ either peripheral or isotopic to $\eta_1$, contrary
to the hypotheses. Hence $S'$ has positive complexity. Therefore there exists an essential simple
closed curve $\sigma\subset S'$ with $i(\eta_0,\sigma)>0$. Since $\sigma$ is contained in the
complement of $\eta_1$, we have $i(\eta_1,\sigma)=0$, contradicting
\[
i(\eta_0,\sigma)=\lambda i(\eta_1,\sigma).
\]
\end{enumerate}
Therefore, the target curves are distinct
for all sufficiently large $k$.
\end{proof}

\subsection{Asymptotics of twist differentials}

Let $Q$ be a one-holed torus with boundary length $L\ge0$,
where $L=0$ corresponds to a cusp. Let $\alpha,\beta$ be
essential simple closed curves with $i(\alpha,\beta)=1$.
Write $a=\ell(\alpha)$, let $\tau$ be the length-valued
twist parameter about $\alpha$ relative to $\beta$, and
set $\beta_m=T_\alpha^m(\beta)$.
Recall from Lemma~\ref{lemma:gluing-identity} that
\[
\ell(\beta_m)
=2\operatorname{arcosh}\left(
F(a,L)\cosh\frac{\tau+ma}{2}
\right),
\]
where $F(a,L)$ and $f(a,L)=\log F(a,L)$ are defined
in \eqref{eq:f_def}. We record the asymptotic expansions
of these lengths and their partial derivatives.

\begin{lemma}\label{lemma:twist-limits}
As $|m|\to\infty$, writing $\operatorname{sgn}(m)$
for the sign of $m$, we have
\begin{align*}
\ell(\beta_m)
&=|m|a+\operatorname{sgn}(m)\tau
+2f(a,L)+O(e^{-|m|a}),\\
\frac{\partial\ell(\beta_m)}{\partial\tau}
&=\operatorname{sgn}(m)+O(e^{-|m|a}),\\
\frac{\partial\ell(\beta_m)}{\partial a}
&=|m|+2\frac{\partial f}{\partial a}
+O(|m|e^{-|m|a}),\\
\frac{\partial\ell(\beta_m)}{\partial L}
&=2\frac{\partial f}{\partial L}
+O(e^{-|m|a}).
\end{align*}
These expansions hold uniformly on any compact set
of parameters $(a,L,\tau)$ with $a>0$ and $L\ge0$.
The derivative with respect to $L$ is understood
for $L>0$.
\end{lemma}

\begin{proof}
Put $u_m=\tau+ma$ and write $\ell_m=\ell(\beta_m)$.
Fix a compact set $K$ of parameters $(a,L,\tau)$ with
$a>0$ and $L\ge0$. All estimates below are uniform on
$K$. Since $a$ is bounded away from zero and $\tau$ is
bounded on $K$, we have $|u_m|\to\infty$ uniformly as
$|m|\to\infty$. Moreover,
\[
\cosh\frac{u_m}{2}
=\frac12 e^{|u_m|/2}(1+e^{-|u_m|}),
\qquad
\cosh^{-2}\frac{u_m}{2}=O(e^{-|u_m|}).
\]
Using
$\operatorname{arcosh}x=\log(2x)+O(x^{-2})$
as $x\to\infty$, the exact length formula gives
\begin{align*}
\ell_m
&=2\log\left(2F(a,L)\cosh\frac{u_m}{2}\right)
+O(e^{-|u_m|})\\
&=|u_m|+2f(a,L)
+2\log(1+e^{-|u_m|})
+O(e^{-|u_m|})\\
&=|u_m|+2f(a,L)+O(e^{-|u_m|}).
\end{align*}

To obtain the derivative expansions, define
\[
R_m=
\left(
1-F(a,L)^{-2}\cosh^{-2}\frac{u_m}{2}
\right)^{-1/2}.
\]
Since $F(a,L)>1$, we have
\[
R_m=1+O(e^{-|u_m|}),
\qquad
\tanh\frac{u_m}{2}
=\operatorname{sgn}(u_m)+O(e^{-|u_m|}).
\]
Differentiating the exact length formula with respect
to $\tau$, $a$, and $L$, respectively, yields
\begin{align*}
\frac{\partial\ell_m}{\partial\tau}
&=R_m\tanh\frac{u_m}{2},\\
\frac{\partial\ell_m}{\partial a}
&=R_m\left(
m\tanh\frac{u_m}{2}
+2\frac{\partial f}{\partial a}
\right),\\
\frac{\partial\ell_m}{\partial L}
&=2\frac{\partial f}{\partial L}R_m.
\end{align*}
Since the first partial derivatives of $f$ are
bounded on $K$, these identities imply
\begin{align*}
\frac{\partial\ell_m}{\partial\tau}
&=\operatorname{sgn}(u_m)+O(e^{-|u_m|}),\\
\frac{\partial\ell_m}{\partial a}
&=m\operatorname{sgn}(u_m)
+2\frac{\partial f}{\partial a}
+O(|m|e^{-|u_m|}),\\
\frac{\partial\ell_m}{\partial L}
&=2\frac{\partial f}{\partial L}
+O(e^{-|u_m|}).
\end{align*}

Finally, since $a$ is bounded away from zero and
$\tau$ is bounded on $K$, for sufficiently large
$|m|$ we have
\[
\operatorname{sgn}(u_m)=\operatorname{sgn}(m),
\qquad
|u_m|=|m|a+\operatorname{sgn}(m)\tau.
\]
Consequently,
\[
e^{-|u_m|}
=e^{-|m|a-\operatorname{sgn}(m)\tau}
=O(e^{-|m|a}),
\]
uniformly on $K$. Substituting these relations into
the length and derivative expansions above gives
the claimed formulas for both $m\to+\infty$
and $m\to-\infty$.
\end{proof}

\section{Proof of the Main Theorem}\label{section-higher-dimensional}

We now prove the higher-dimensional result for every finite-type surface of positive genus. Cusps are allowed, and geodesic boundary components are allowed when their lengths are fixed. All curves below are essential interior simple closed curves.

\begin{proof}[Proof of Theorem \ref{thm:main_intro}]
Fix $X\in\mathcal T_{\mathbf L}(S)$ and an open neighborhood $U$ of $X$. We will find a surface with the required curves inside $U$.

\subsection*{1. Choosing $d+1$ base curves.}

For every essential interior simple closed curve $\gamma$, let $\ell(\gamma) : \mathcal{T}_{\mathbf L}(S) \to \mathbb{R}_{>0}$ denote the geodesic length function. We define the covector:
$$
g_\gamma=d\log\ell(\gamma)\big|_X\in T_X^*\mathcal T_{\mathbf L}(S).
$$
We claim that the covectors $g_\gamma$ affinely
span $T_X^*\mathcal T_{\mathbf L}(S)$.
Otherwise, fixing an essential interior simple
closed curve $\eta_0$, there exists a nonzero
$v\in T_X\mathcal T_{\mathbf L}(S)$ such that
\[
\langle g_\gamma-g_{\eta_0},v\rangle=0
\]
for every essential interior simple closed curve
$\gamma$. Setting $c=\langle g_{\eta_0},v\rangle$
and choosing a smooth path $X_t$ with
$X_0=X$ and $\dot X_0=v$, we obtain
\[
\dot\ell(\gamma)
=\left.\frac{d}{dt}\right|_{t=0}
\ell_{X_t}(\gamma)
=c\ell_X(\gamma)
\]
for every such curve $\gamma$.

Choose a one-holed-torus subsurface $Q$ carrying
one handle of $S$, allowing a cusp in place of
a geodesic boundary. Let $\delta$ denote its
boundary. If $Q\ne S$, choose $Q$ so that
$\delta$ is an essential separating interior curve
and all remaining genus and ends lie in the
complementary component. If $Q=S$, then
$\delta$ is the geodesic boundary or cusp of $S$.
We first use the geometry of $Q$ to show that
$c=0$, and then use Fenchel-Nielsen
coordinates to conclude that $v=0$. 

\subsubsection*{Step 1.1. Asymptotic twist dynamics}

Choose simple curves $\alpha,\beta\subset Q$ with $i(\alpha,\beta)=1$. To justify our local coordinate calculations, choose the Fenchel-Nielsen system so that $\alpha$ and, when interior, $\delta$ belong to a pants decomposition. Put $a=\ell_X(\alpha)$, and let $\tau$ be the length-valued twist about $\alpha$, measured using $\beta$. If $\delta$ is an interior curve or a geodesic boundary component, put $L=\ell_X(\delta)$. If $\delta$ is a cusp, put $L=0$.

The intrinsic hyperbolic metric on $Q$ is
determined by $(a,\tau,L)$. If $Q\ne S$,
its geodesic boundary ensures that geodesic
representatives of non-peripheral curves carried
by $Q$ remain in $Q$. Thus their lengths depend
only on $(a,\tau,L)$. 
Consider $\beta_m=T_\alpha^m(\beta)$.
By Lemma~\ref{lemma:twist-limits}, as
$m\to+\infty$ we have
\begin{align*}
\ell(\beta_m) &= ma + \tau + 2f(a,L) + O(e^{-ma}), \\
\frac{\partial \ell(\beta_m)}{\partial \tau} &= 1 + O(e^{-ma}), \\
\frac{\partial \ell(\beta_m)}{\partial a} &= m + 2\frac{\partial f}{\partial a} + O(me^{-ma}), \\
\frac{\partial \ell(\beta_m)}{\partial L} &= 2\frac{\partial f}{\partial L} + O(e^{-ma}),
\end{align*}
where $f(a,L)$ is the logarithmic seam factor defined in Equation \eqref{eq:f_def}. 
The time derivative along the path $X_t$ at $t=0$ is
$$
\dot{\ell}(\beta_m) = \dot{a}\frac{\partial \ell(\beta_m)}{\partial a} + \dot{\tau}\frac{\partial \ell(\beta_m)}{\partial \tau} + \dot{L}\frac{\partial \ell(\beta_m)}{\partial L}.
$$
(If $L=0$, the deformation space has no boundary-length direction and $\dot{L}=0$, so this third term simply does not appear.) 
Substituting the asymptotic limits into this chain rule yields:
\begin{align*}
\dot{\ell}(\beta_m) &= \dot{a}\Big(m + 2\frac{\partial f}{\partial a} + O(me^{-ma})\Big) + \dot{\tau}\Big(1 + O(e^{-ma})\Big) + \dot{L}\Big(2\frac{\partial f}{\partial L} + O(e^{-ma})\Big) \\
&= m\dot{a} + \dot{\tau} + 2\left(\dot{a}\frac{\partial f}{\partial a} + \dot{L}\frac{\partial f}{\partial L}\right) + O(me^{-ma}) \\
&= m\dot{a} + \dot{\tau} + 2\dot{f} + O(me^{-ma}).
\end{align*}
Since
$\dot a=ca$ and $\dot\ell(\beta_{m})=c\ell(\beta_{m})$,
$$
mca + \dot{\tau} + 2\dot{f} + O(me^{-ma}) = c\ell(\beta_m).
$$
Substituting the length expansion $\ell(\beta_m) = ma + \tau + 2f + O(e^{-ma})$ into the right side yields:
$$
mca + \dot{\tau} + 2\dot{f} + O(me^{-ma}) = c\Big(ma + \tau + 2f + O(e^{-ma})\Big)
$$
which implies
$$
\dot{\tau} + 2\dot{f} + O(me^{-ma}) = c(\tau + 2f) + cO(e^{-ma}).
$$
Because $a > 0$, letting $m \to \infty$ forces all error terms to $0$, yielding the exact equation:
$$
\dot\tau+2\dot f = c(\tau+2f).
$$
 
Applying the negative-twist expansions of Lemma \ref{lemma:twist-limits} to
$\beta_{-m}$, with $m\to+\infty$, yields
$$
-\dot\tau+2\dot f = c(-\tau+2f).
$$
Therefore,
$$
4\dot{f} = 4cf \implies \dot f=cf.
$$

\subsubsection*{Step 1.2. Forcing $c=0$.} 
We distinguish two cases according to the
boundary $\delta$ of $Q$.

\emph{Case 1: $\delta$ is an essential interior curve.} 
Because $\delta$ is an interior curve, its length derivative is $\dot{L} = cL$. Recall that $\dot{a} = ca$.
We have
$$
\dot{f} = \dot{a}\frac{\partial f}{\partial a} + \dot{L}\frac{\partial f}{\partial L}
= c\left(a\frac{\partial f}{\partial a} + L\frac{\partial f}{\partial L}\right).
$$
Substituting $\dot f=cf$ gives
\[
c\left(a\frac{\partial f}{\partial a}
+L\frac{\partial f}{\partial L}-f\right)=0.
\]
It therefore suffices to show that
\[
a\frac{\partial f}{\partial a}
+L\frac{\partial f}{\partial L}-f<0.
\]
To this end, set
\[
p=\frac{a}{2},\qquad q=\frac{L}{4} 
\implies a = 2p, \quad \frac L2 = 2q.
\]
We have
\begin{align*}
f &= \log F(a,L) \\
&= \log \sqrt{\frac{\cosh a + \cosh(L/2)}{\cosh a - 1}} \\
&= \frac{1}{2}\log\left( \frac{\cosh(2p) + \cosh(2q)}{\cosh(2p) - 1} \right) \\
&= \frac{1}{2}\log\left( \frac{2\cosh(p+q)\cosh(p-q)}{2\sinh^2 p} \right) \\
&= \frac{1}{2}\Big( \log\cosh(p+q) + \log\cosh(p-q) - 2\log\sinh p \Big).
\end{align*}
Since $a = 2p$ and $L = 4q$,
$$
a\frac{\partial}{\partial a} + L\frac{\partial}{\partial L} = 2p\left(\frac{1}{2}\frac{\partial}{\partial p}\right) + 4q\left(\frac{1}{4}\frac{\partial}{\partial q}\right) = p\frac{\partial}{\partial p} + q\frac{\partial}{\partial q}.
$$
We apply this operator to our expanded formula for $f$:
\begin{align*}
p\frac{\partial f}{\partial p} + q\frac{\partial f}{\partial q}
&= \frac{1}{2}\Big( p\tanh(p+q) + p\tanh(p-q) - 2p\coth p \Big) \\
&\quad + \frac{1}{2}\Big( q\tanh(p+q) - q\tanh(p-q) \Big) \\
&= \frac{1}{2}\Big( (p+q)\tanh(p+q) + (p-q)\tanh(p-q) - 2p\coth p \Big).
\end{align*}
Therefore,
\begin{align*}
a\frac{\partial f}{\partial a} + L\frac{\partial f}{\partial L} - f
&= \frac{1}{2} \Big[ (p+q)\tanh(p+q) - \log\cosh(p+q) \Big] \\
&\quad + \frac{1}{2} \Big[ (p-q)\tanh(p-q) - \log\cosh(p-q) \Big] \\
&\quad - \Big[ p\coth p - \log\sinh p \Big].
\end{align*}
To determine its sign, define
$$
\phi(x) = x\tanh x - \log\cosh x, \qquad \psi(p) = p\coth p - \log\sinh p.
$$
Our expression simplifies into:
$$
a\frac{\partial f}{\partial a} + L\frac{\partial f}{\partial L} - f = \frac{1}{2}\phi(p+q) + \frac{1}{2}\phi(p-q) - \psi(p).
$$
We next bound \(\phi\) and \(\psi\). For $x > 0$, taking the derivative gives:
$$
\phi'(x) = \frac{x}{\cosh^2 x} > 0, \qquad \lim_{x\to\infty}\phi(x) = \log 2.
$$
Since $\phi$ is even and $\phi(0)=0$,
we conclude that $\phi(x)<\log 2$
for every $x\in\mathbb R$. 
Since $p > 0$,
$$
\psi'(p) = -\frac{p}{\sinh^2 p} < 0, \qquad \lim_{p\to\infty}\psi(p) = \log 2.
$$
Since $\psi(p)$ is strictly decreasing and bounded below by $\log 2$, we have $\psi(p) > \log 2$ for all $p > 0$. 
Consequently, 
\begin{align*}
a\frac{\partial f}{\partial a} + L\frac{\partial f}{\partial L} - f
&= \frac{1}{2}\phi(p+q) + \frac{1}{2}\phi(p-q) - \psi(p) \\
&< \frac{1}{2}(\log 2) + \frac{1}{2}(\log 2) - \log 2 \\
&= 0.
\end{align*}
Thus $c=0$.

\emph{Case 2: $\delta$ is a cusp or a fixed geodesic boundary component.} 
In this case $Q=S$ and $\dot L=0$. 
Together with $\dot a=ca$, the chain rule gives
$$
\dot f=\dot a\frac{\partial f}{\partial a}
=ca\frac{\partial f}{\partial a}.
$$
Substituting $\dot f=cf$ yields
$$
c\left(a\frac{\partial f}{\partial a}-f\right)=0.
$$
Now
$$
\frac{\partial f}{\partial a}
=
-\frac{(1+\cosh(L/2))\sinh a}
{2(\cosh a+\cosh(L/2))(\cosh a-1)}
<0.
$$
Moreover,
$$
f(a,L)>0.
$$
Therefore,
$$
a\frac{\partial f}{\partial a}-f<0,
$$
which forces $c=0$.

\subsubsection*{Step 1.3. Forcing $v=0$.}
We now prove that $v=0$. Choose a pants decomposition
$$
P=\{\alpha_1,\ldots,\alpha_q\},
\qquad
q=3g-3+n+b.
$$
Since the boundary lengths are fixed, Fenchel-Nielsen coordinates on $\mathcal T_{\mathbf L}(S)$ are
$$
(\ell(\alpha_1),\ldots,\ell(\alpha_q),
\tau_{\alpha_1},\ldots,\tau_{\alpha_q}).
$$
By Step 1.2, we have 
$$
d\ell(\alpha_i)(v) = \left. \frac{d}{dt} \right|_{t=0} \ell_{X_t}(\alpha_i) = \dot{\ell}(\alpha_i) = c \ell_X(\alpha_i) = 0
$$
for every $i$, so all length components of $v$ vanish.
Hence
$$
v=\sum_{i=1}^q c_i\frac{\partial}{\partial\tau_{\alpha_i}}.
$$  
It remains to show $c_i=0$ for every $i$. For each $\alpha_i$, we choose an essential interior simple closed curve $\beta_i$ disjoint from every $\alpha_j$ with $j\ne i$. We distinguish two cases based on the local topology around $\alpha_i$.

\emph{Case 1: The two sides of $\alpha_i$ belong to distinct adjacent pairs of pants.} 
In this case, let $\beta_i$ be the standard dual curve in the union of these two pairs of pants, intersecting $\alpha_i$ twice. Consider
$$
\beta_{i,m}=T_{\alpha_i}^m(\beta_i).
$$
Let $\theta_i(X,s)$ denote the surface obtained by applying a Fenchel-Nielsen twist to $X$ along $\alpha_i$, parameterized such that its length-valued twist coordinate is
$$
\tau_{\alpha_i}(s) = \tau_{\alpha_i}(X) + s\ell_X(\alpha_i).
$$
Consequently, by the chain rule, taking the partial derivative with respect to $\tau_{\alpha_i}$ corresponds to applying the operator:
$$
\frac{\partial}{\partial\tau_{\alpha_i}} = \frac{1}{\ell_X(\alpha_i)} \frac{d}{ds}.
$$
We define $h_i(s) = \ell_{\theta_i(X,s)}(\beta_i)$
to be the length of the transversal curve $\beta_i$
as a function of the twist parameter $s$.
By \cite[Proposition~7.3 and Corollary~7.4(2)]{MR3053012},
$h_i(s)$ is strictly convex, attains a unique minimum,
and tends to $+\infty$ as $s\to\pm\infty$.
Although stated for closed surfaces, the proof uses
only the hyperbolic geometry of the adjacent-pants
configuration and applies equally to surfaces with
cusps or geodesic boundary.
Consequently, $h_i'(s)\ne0$ for all sufficiently
large $|s|$.

By mapping-class equivariance (\cite[Remark 2, Expos\'e 7]{MR3053012}), applying $m$ Dehn twists to a curve is geometrically equivalent to adding $m$ full turns to the surface's twist parameter. Thus, the length of the twisted curve $\beta_{i,m}$ on the surface $\theta_i(X,s)$ is exactly the length of $\beta_i$ on $\theta_i(X,s+m)$:
$$
\ell_{\theta_i(X,s)}(\beta_{i,m}) = h_i(s+m).
$$
Evaluating the partial derivative at the base surface $X$ (where $s=0$) yields:
\begin{align*}
\left. \frac{\partial \ell(\beta_{i,m})}{\partial\tau_{\alpha_i}} \right|_X 
&= \frac{1}{\ell_X(\alpha_i)} \left. \frac{d}{ds} \Big( h_i(s+m) \Big) \right|_{s=0} \\
&= \frac{h_i'(m)}{\ell_X(\alpha_i)}.
\end{align*}
Since $h_i'(m)\ne0$ for all sufficiently
large $|m|$, this partial derivative is nonzero.

\emph{Case 2: The two sides of $\alpha_i$ occur on the same pair of pants.} 
Gluing the two corresponding boundary components produces a one-holed-torus subsurface. We choose $\beta_i$ in that torus with $i(\alpha_i,\beta_i)=1$. Using Lemma \ref{lemma:twist-limits}, the partial derivative $\left. \frac{\partial \ell(\beta_{i,m})}{\partial\tau_{\alpha_i}} \right|_X$ approaches $\pm 1$ as $m \to \pm\infty$, so the derivative is nonzero.

In summary, for either case, one may choose $m$ with
$$
\left. \frac{\partial \ell(\beta_{i,m})}{\partial\tau_{\alpha_i}} \right|_X \ne 0.
$$
Because $\beta_{i,m}$ is disjoint from every other pants curve $\alpha_j$, the sum over intersection points in Wolpert's cosine formula \cite{MR615620} is empty. Consequently, we have
$$
\left. \frac{\partial \ell(\beta_{i,m})}{\partial\tau_{\alpha_j}} \right|_X 
= \sum_{p \in \alpha_j \cap \beta_{i,m}} \cos \theta_p 
= 0
\qquad
(j\ne i),
$$
where $\theta_p$ denotes the intersection angle at a point $p$. 
It follows that
$$
0
=
d\ell(\beta_{i,m})(v)
=
\sum_{j=1}^q c_j \left. \frac{\partial \ell(\beta_{i,m})}{\partial\tau_{\alpha_j}} \right|_X 
=
c_i \left. \frac{\partial \ell(\beta_{i,m})}{\partial\tau_{\alpha_i}} \right|_X,
$$
and since the partial derivative is non-zero, we must have $c_i=0$.  
Since this holds for an arbitrary $i$, all twist components of $v$ also vanish. Therefore $v=0$, a contradiction.
 
Therefore the covectors $g_\gamma$ have full affine span, that is, there are $d+1$ distinct essential interior simple closed curves
$$
\eta_0,\eta_1,\ldots,\eta_d
$$
such that
$$
g_{\eta_1}-g_{\eta_0},
\ldots,
g_{\eta_d}-g_{\eta_0}
$$
are linearly independent. Define
$$
H(Y)
=
\left(
\log\frac{\ell_Y(\eta_1)\ell_X(\eta_0)}{\ell_Y(\eta_0)\ell_X(\eta_1)},
\ldots,
\log\frac{\ell_Y(\eta_d)\ell_X(\eta_0)}{\ell_Y(\eta_0)\ell_X(\eta_d)}
\right).
$$
Then $H(X)=0$ and $dH_X$ is invertible. By the inverse function theorem, there is an open neighborhood $V \subset U$ of $X$ on which $H: V \longrightarrow H(V)$ is a diffeomorphism. 
Because $H(V)$ is open and contains $H(X)=0$, we can choose $\varepsilon>0$ such that the closed Euclidean ball $\overline{D_\varepsilon}$ centered at the origin is contained in $H(V)$. Defining $B=(H|_V)^{-1}(\overline{D_\varepsilon})$, we obtain a compact neighborhood $B \subset V \subset U$ of $X$ such that the restriction
$$
H:B\longrightarrow\overline{D_\varepsilon}
$$
is a homeomorphism and $H(\partial B) = \partial D_\varepsilon$.

\subsection*{2. The target curves}

For each $j\in\{0,\ldots,d\}$, choose an essential interior simple closed curve $\xi_j$ such that the geometric intersection number is strictly positive:
$$
I_j=i(\eta_j,\xi_j)>0.
$$
For every positive integer $k$, we calibrate the twist frequency by setting
$$
N_j(k) = \left\lfloor \frac{k r_j}{I_j\ell_X(\eta_j)} \right\rfloor = \frac{k r_j}{I_j\ell_X(\eta_j)} - \delta_j(k),
$$
where $0 \le \delta_j(k) < 1$ accounts for the bounded error from the integer floor function. We assume $k$ is sufficiently large so that every $N_j(k) \ge 1$. We then define the target twisted curves as
$$
\gamma_j^{(k)} = T_{\eta_j}^{N_j(k)}(\xi_j).
$$
Applying Lemma \ref{lemma:distinct_targets} to each of the finitely many pairs $(\eta_j, \eta_m)$ yields a $k_0$ such that for all $k \ge k_0$, the target curves $\gamma_0^{(k)}, \ldots, \gamma_d^{(k)}$ are pairwise distinct.

For each $j$ and every $Y\in B$,
Lemma~\ref{lemma:universal_collar_squeeze} gives
$$
\ell_Y(\gamma_j^{(k)}) = N_j(k) \cdot I_j \cdot \ell_Y(\eta_j) + E_j(k,Y),
$$
where the error function $E_j(k,Y)$ satisfies $|E_j(k,Y)| \le C_j$ for some constant $C_j$ depending only on $B, \eta_j,$ and $\xi_j$. Substituting the expression for $N_j(k)$ yields:
\begin{align*} 
\ell_Y(\gamma_j^{(k)}) &= \left( \frac{k r_j}{I_j\ell_X(\eta_j)} - \delta_j(k) \right) I_j \ell_Y(\eta_j) + E_j(k,Y) \\
&= k r_j\frac{\ell_Y(\eta_j)}{\ell_X(\eta_j)} + \Big( E_j(k,Y) - \delta_j(k) I_j \ell_Y(\eta_j) \Big).
\end{align*}
We define the aggregate remainder function as 
$$
\mathcal{E}_j(k,Y) = E_j(k,Y) - \delta_j(k) I_j \ell_Y(\eta_j).
$$

Since $B$ is compact and $0\le\delta_j(k)<1$,
the definition of $\mathcal{E}_j(k,Y)$ and the
uniform bound on $E_j(k,Y)$ imply that
$|\mathcal{E}_j(k,Y)|$ is bounded independently
of $k$ and $Y\in B$. As there are only finitely
many indices $j$, there exists $C>0$ such that
\[
|\mathcal{E}_j(k,Y)|\le C
\]
for every $j=0,\ldots,d$, every $Y\in B$,
and all sufficiently large $k$. Thus,
\[
\ell_Y(\gamma_j^{(k)})
=kr_j\frac{\ell_Y(\eta_j)}{\ell_X(\eta_j)}
+\mathcal{E}_j(k,Y),
\]
where $\mathcal{E}_j(k,Y)/k\to 0$ uniformly
on $B$ as $k\to\infty$.

\subsection*{3. Solving the exact ratio equations.}
Define the target map $G_k : B \to \mathbb{R}^d$ by
$$
G_k(Y) = \left( \log\frac{\ell_Y(\gamma_1^{(k)})r_0}{\ell_Y(\gamma_0^{(k)})r_1}, \ldots, \log\frac{\ell_Y(\gamma_d^{(k)})r_0}{\ell_Y(\gamma_0^{(k)})r_d} \right).
$$
Using the length formula from Step 2, we obtain
\begin{align*}
\frac{\ell_Y(\gamma_j^{(k)})r_0}{\ell_Y(\gamma_0^{(k)})r_j} 
&= \frac{\left( k r_j \frac{\ell_Y(\eta_j)}{\ell_X(\eta_j)} + \mathcal{E}_j(k,Y) \right) r_0}{\left( k r_0 \frac{\ell_Y(\eta_0)}{\ell_X(\eta_0)} + \mathcal{E}_0(k,Y) \right) r_j} \\
&= \frac{r_j \frac{\ell_Y(\eta_j)}{\ell_X(\eta_j)} + \frac{\mathcal{E}_j(k,Y)}{k}}{r_0 \frac{\ell_Y(\eta_0)}{\ell_X(\eta_0)} + \frac{\mathcal{E}_0(k,Y)}{k}} \cdot \frac{r_0}{r_j}.
\end{align*}
For every $j=0,\ldots,d$, define the leading continuous term $A_j(Y)=r_j\frac{\ell_Y(\eta_j)}{\ell_X(\eta_j)}$. Because the domain $B$ is compact and length functions are strictly positive, we have $\min_{Y\in B}A_j(Y) > 0$ for every $j$.  
As established above, the scaled error terms $\mathcal{E}_j(k,Y)/k \to 0$ uniformly on $B$ as $k \to \infty$. 
Consequently, for all sufficiently large $k$,
these expressions are bounded away from zero
uniformly on $B$. By continuity of the logarithm,
$G_k$ converges uniformly to $H$ on $B$. 

By our construction, $H$ is a homeomorphism from $B$ onto the closed Euclidean ball $\overline{D_\varepsilon}$ of radius $\varepsilon$ centered at the origin. In particular, $H$ maps the boundary $\partial B$ bijectively onto the sphere $\partial D_\varepsilon$. Therefore, the Euclidean distance from the origin to any point on $H(\partial B)$ is exactly $\varepsilon$.

By the uniform convergence of $G_k$ to $H$, for all
sufficiently large $k$, we have
\[
\|G_k(Y)-H(Y)\|<\frac{\varepsilon}{2}
\]
for every $Y\in B$.

Define
\[
\widetilde G_k=G_k\circ H^{-1}:
\overline{D_\varepsilon}\longrightarrow\mathbb R^d.
\]
Then
\[
\|\widetilde G_k(z)-z\|<\frac{\varepsilon}{2}
\]
for every $z\in\overline{D_\varepsilon}$.
Consequently, the continuous map
\[
T_k(z)=z-\widetilde G_k(z)
\]
maps $\overline{D_\varepsilon}$ into itself.
By the Brouwer fixed-point theorem, there exists
$z_k\in\overline{D_\varepsilon}$ such that
$T_k(z_k)=z_k$, or equivalently,
$\widetilde G_k(z_k)=0$.

Setting $Y_k=H^{-1}(z_k)$, we obtain
$G_k(Y_k)=0$. By the definition of $G_k$,
this gives
\[
\frac{\ell_{Y_k}(\gamma_j^{(k)})}{r_j}
=
\frac{\ell_{Y_k}(\gamma_0^{(k)})}{r_0}
\]
for every $j=1,\ldots,d$.
For sufficiently large $k$, the curves
$\gamma_0^{(k)},\ldots,\gamma_d^{(k)}$
are pairwise distinct. Since $Y_k\in B\subset U$,
we have found a surface satisfying the prescribed
length ratios in the arbitrary neighborhood $U$.
This completes the proof.
\end{proof}

\begin{remark}\label{rem:boundary-cases}
Theorem~\ref{thm:main_intro} applies to positive-genus
surfaces with cusps and geodesic boundary components
of arbitrary fixed lengths. The proof uses only
essential interior simple closed curves. 
If the geodesic boundary lengths are allowed to vary,
the dimension of Teichm\"uller space increases by $b$,
introducing boundary-length directions not controlled
by our argument. This does not rule out an analogous
result for positive-genus surfaces with variable
boundary lengths. 
For genus-zero surfaces, our proof does not apply
because the one-holed-torus subsurface used in Step 1
is unavailable.  
\end{remark}

\section{Density Along Coordinate Lines}
\label{section-coordinate-lines}

We conclude with a complementary result for
the one-holed torus: prescribed length ratios
are dense along both constant-length and
constant-twist lines.

\begin{theorem}\label{thm:one-dimensional-slices}
Let $S$ be a one-holed torus with fixed boundary length
$L\ge0$, and let $(a,\tau)$ be the Fenchel-Nielsen
coordinates associated to curves $\alpha,\beta$ with
$i(\alpha,\beta)=1$.

For every $q>0$, the surfaces admitting two distinct
essential interior simple closed geodesics with length
ratio $q$ form a dense subset of each constant-length
line $a=a_0$ and each constant-twist line $\tau=\tau_0$,
where $a_0>0$ and $\tau_0\in\mathbb R$.
\end{theorem}

\begin{proof}
Fix a point $X=(a_0,\tau_0)$ on one of the two
coordinate lines and an arbitrary neighborhood $U$
of $X$ within that line. Parameterize a small interval
by $c(t)$ with $c(0)=X$. We first choose two distinct
base curves $\eta_0,\eta_1$ such that
\[
H(t)=
\log\frac{\ell_{c(t)}(\eta_1)\ell_X(\eta_0)}
{\ell_{c(t)}(\eta_0)\ell_X(\eta_1)}
\]
satisfies $H'(0)\ne0$.

On the constant-twist line
$c(t)=(a_0+t,\tau_0)$, take
$\eta_0=\alpha$ and $\eta_1=\beta$. Then
\[
\frac{d}{da}\log\ell(\alpha)=\frac1a>0.
\]
Moreover, by Lemma~\ref{lemma:gluing-identity},
\[
\cosh\frac{\ell(\beta)}2
=F(a,L)\cosh\frac{\tau_0}{2},
\]
and
\[
\frac{\partial}{\partial a}\log F(a,L)
=
-\frac{(1+\cosh(L/2))\sinh a}
{2(\cosh a+\cosh(L/2))(\cosh a-1)}<0.
\]
Since $\operatorname{arcosh}$ is increasing,
we have $\partial_a\ell(\beta)<0$.
Hence $H'(0)\ne0$.

On the constant-length line
$c(t)=(a_0,\tau_0+t)$, take $\eta_0=\alpha$
and choose $n\in\mathbb Z$ with
$\tau_0+na_0\ne0$. Set
$\eta_1=T_\alpha^n(\beta)$.
The length of $\eta_0$ is constant on this
line, whereas Lemma~\ref{lemma:gluing-identity}
gives
\[
\left.
\frac{\partial}{\partial\tau}
\ell(T_\alpha^n\beta)\right|_X
=
\frac{F(a_0,L)\sinh((\tau_0+na_0)/2)}
{\sqrt{F(a_0,L)^2
\cosh^2((\tau_0+na_0)/2)-1}}
\ne0.
\]
Thus again $H'(0)\ne0$. 
Since $H(0)=0$ and $H'(0)\ne0$, we can
choose $\varepsilon>0$ such that
$c([-\varepsilon,\varepsilon])\subset U$
and $H(-\varepsilon)$ and $H(\varepsilon)$
have opposite signs.

Write $q=r_1/r_0$ with $r_0,r_1>0$.
For $i=0,1$, choose an essential interior
simple closed curve $\xi_i$ with
$I_i=i(\xi_i,\eta_i)>0$, and set
\[
N_i(k)=\left\lfloor
\frac{kr_i}{I_i\ell_X(\eta_i)}
\right\rfloor,
\qquad
\gamma_i^{(k)}
=T_{\eta_i}^{N_i(k)}(\xi_i).
\]
By Lemma~\ref{lemma:distinct_targets},
the curves $\gamma_0^{(k)}$ and
$\gamma_1^{(k)}$ are distinct for all
sufficiently large $k$.
By Lemma~\ref{lemma:universal_collar_squeeze},
uniformly for $t\in[-\varepsilon,\varepsilon]$,
\[
\ell_{c(t)}(\gamma_i^{(k)})
=
kr_i\frac{\ell_{c(t)}(\eta_i)}
{\ell_X(\eta_i)}+O(1).
\]
Consequently, the maps
\[
G_k(t)=
\log\frac{\ell_{c(t)}(\gamma_1^{(k)})r_0}
{\ell_{c(t)}(\gamma_0^{(k)})r_1}
\]
converge uniformly to $H(t)$.
For all sufficiently large $k$,
$G_k(-\varepsilon)$ and $G_k(\varepsilon)$
therefore have opposite signs.
By the intermediate value theorem,
there exists $t_k\in(-\varepsilon,\varepsilon)$
such that $G_k(t_k)=0$. Hence
\[
\frac{\ell_{c(t_k)}(\gamma_1^{(k)})}
{\ell_{c(t_k)}(\gamma_0^{(k)})}
=\frac{r_1}{r_0}=q.
\]
Since $U$ was arbitrary, this proves density
on the chosen line.
\end{proof}

\subsection*{AI Use}
ChatGPT (OpenAI) assisted with checking the proofs
and proofreading the paper. The authors take full
responsibility for the mathematical content.

\bibliographystyle{plain}
\bibliography{ref}

\end{document}